\documentclass[10pt]{article}
\usepackage{fullpage,enumitem}
\usepackage[usenames]{color}  
\usepackage[colorlinks = true]{hyperref}
\usepackage[english]{babel}
\usepackage{bbm,amsmath,amssymb,amsthm,amsbsy,mathtools}
\usepackage{cite,graphicx,url,thm-restate}
\usepackage[title]{appendix}
\numberwithin{equation}{section}
\newcommand  \tmop[1]    {{\ensuremath{\operatorname{#1}}}}
\newcommand  \stack[2]   {\overset{(\text{#1})}{#2}}

\newcommand{\E}     {\tmop{E}}

\def \x {\boldsymbol{x}}

\def \minimize  {\operatorname*{minimize}}

\def \argmin    {\operatorname*{argmin}}

\def \st        {\operatorname*{subject\ to\ }}

\def \R {\mathbb{R}}
\def \C {\mathbb{C}}

\def \O {\mathbb{O}}

\def \eye       {\mathbf{I}}
\def \zero      {\mathbf{0}}

\def \lg        {\langle}
\def \rg        {\rangle}

\def \tr        {\tmop{trace}}
\def \rank      {\tmop{rank}}

\def \range     {\tmop{Range}}

\newtheorem{proposition}{Proposition}
\newtheorem{lem}{Lemma}
\newtheorem{theorem}{Theorem}

\newtheorem{definition}{Definition}
\theoremstyle{remark}
\newtheorem{remark}{Remark}

\def \Poff {\mathcal{P_{\tmop{off}}}}
\def \Pon  {\mathcal{P_{\tmop{on}}}}
\def \dist {\tmop{d}}
\def \wh   {\widehat}

\def \PP   {\mathbb{P}}
\def \X    {\mathcal{X}}
\def \E    {\mathcal{E}}
\begin{document}
\title{Geometry of Factored Nuclear Norm Regularization}
 \author{\normalsize Qiuwei Li, Zhihui Zhu and Gongguo Tang\thanks{ This work was supported by NSF grants CCF-1464205, CCF-1409261. Email: \{qiuli,gtang,zzhu\}@mines.edu}\\
\normalsize Department of Electrical Engineering and Computer Science, \\
\normalsize Colorado School of Mines,
Golden, CO 80401, USA
}

\date{\normalsize \today}
\maketitle

\begin{abstract}
This work investigates the geometry of a nonconvex reformulation of minimizing a general convex loss function $f(X)$ regularized by the matrix nuclear norm $\|X\|_*$.
Nuclear-norm regularized matrix inverse problems are at the heart of many applications in  machine learning, signal processing, and control. The statistical performance of nuclear norm regularization has been studied extensively in literature using convex analysis techniques.
Despite its optimal performance, the resulting optimization has high computational complexity when solved using standard or even tailored fast convex solvers.
To develop faster and more scalable algorithms, we follow the proposal of Burer-Monteiro to factor  the matrix variable $X$ into the product of two smaller rectangular matrices $X=UV^T$ and also replace the nuclear norm $\|X\|_*$ with $(\|U\|_F^2+\|V\|_F^2)/2$. In spite of the nonconvexity of the factored formulation, we prove that when the convex loss function $f(X)$ is $(2r,4r)$-restricted well-conditioned, each critical point of the factored problem either corresponds to the optimal solution $X^\star$ of the original convex optimization or is a strict saddle point where the Hessian matrix  has a strictly negative eigenvalue. Such a geometric structure of the factored formulation allows many local search algorithms to converge to the global optimum with random initializations.
\end{abstract}

\section{Introduction} \label{sec:introduction}
Nuclear-norm regularized inverse problems arise in many applications in machine learning \cite{harchaoui2012large}, signal processing \cite{bouwmans2016handbook}, and control \cite{mohan2010reweighted}.
In this work, we consider a general nuclear-norm regularized optimization:
\begin{align}\label{eqn:origin}
\minimize_{X\in\R^{p\times q}} f(X)+\lambda\|X\|_*.
\end{align}
Here $f(X)$ is a general convex loss function, $\|X\|_*$ denotes the matrix nuclear norm of $X$, and $\lambda$ is a trade-off parameter.
The statistical performance has been studied extensively in literature using convex analysis techniques \cite{candes2009exact}, for example, information-theoretically optimal sampling complexity \cite{candes2010power},  minimax denoising rate \cite{candes2010matrix}, and tight oracle inequalities \cite{candes2011tight}. In spite of its optimal performance, improving computational efficiency for \eqref{eqn:origin} remains a challenge. Even fast first-order methods, such as the projected gradient descent  algorithm \cite{boyd2004convex,chen2015fast}, require an expensive singular value decomposition in each iteration, forming the major computational bottleneck of the algorithms and thus preventing them from scaling to big-data applications \cite{decoste2006collaborative}.

\subsection{Our Approach: Burer-Monteiro Parameterization}
To overcome the computational challenges, we utilize the Burer-Monteiro parameterization that is recognized as an alternative to convex solvers in \cite{burer2003nonlinear}. More precisely, when \eqref{eqn:origin} admits a solution $X^\star$ with rank $r^\star$,  the matrix variable is decomposed as the product of two smaller matrices:
\begin{align}\label{eqn:bilinear}
X=\phi(U,V):=U V^{ T},
\end{align}
where $U\in\R^{p\times r}$ and $V\in\R^{q\times r}$ with $r\ll \min\{p,q\}$ and $r \geq r^\star$. Moreover, using the fact that \cite[page 21]{recht2010guaranteed}
\begin{align}\label{eqn:nuclear:nvx}
\|X\|_*=\minimize_{X=\phi(U,V)}\ \Theta(U,V),
\end{align}
where $\Theta(U,V)=\frac{1}{2}(\|U\|_F^2+\|V\|_F^2)$, we replace the matrix nuclear norm $\|X\|_*$ with $\Theta(U,V)$ to obtain a factored formulation of the original convex optimization \eqref{eqn:origin}:
\begin{align}\label{eqn:factored}
\hspace{-0.4cm} \minimize_{U\in\R^{p\times r},V\in\R^{q\times r}} g(U,V):=f(\phi(U,V))+ \lambda \Theta(U,V).
\end{align}
For simplicity, $g(U,V)$ in \eqref{eqn:factored} is also represented by $g(W)$ with $W=[U^T~V^T]^T$.
The factored formulation \eqref{eqn:factored}, coupled with local search algorithms such as gradient descent and its variants, reduces computational complexity by avoiding  expensive SVDs and  decreasing the number
of optimization variables from $pq$ to $(p+q)r$, a significant reduction when $r\ll \min\{p,q\}$. Such an increase in  computational efficiency makes it possible to handle problems with millions of variables.

Although the Burer-Monteiro factored reformulation \eqref{eqn:factored}  allows faster implementations,
the theoretical performance guarantees for the nuclear-norm regularization are not compete. In this work, we adopt the geometric approach \cite{ge2015escaping,ge2016matrix,sun2015complete,sun2016geometric, zhu2017globalb,zhu2017globala,jin2017escape,Ge2017local} to analyze the landscape of \eqref{eqn:factored} and show that there is no spurious local minima or degenerate saddle points for \eqref{eqn:factored} when the objective function $f(X)$ is well-conditioned.
An important implication is that local search algorithms, such as gradient descent and its variants, are able to converge to the global optima even with random initialization \cite{ge2015escaping,ge2016matrix}.
Moreover, once we establish the equivalence between the convex and the factored formulations, it is unnecessary to rederive the statistical performances of the factored optimization \eqref{eqn:factored}, since they inherit from that of the convex optimization \eqref{eqn:origin}.

\subsection{Main Result}
Before presenting our main result, we provide several necessary definitions. We call a vector $\x$ a \emph{critical point} of some differentiable function $\psi(\cdot)$ if the gradient $\nabla \psi(\x)=\zero.$ When $\psi(\cdot)$ is twice continuously differentiable, a critical point $\x$ is called a \emph{strict saddle} or \emph{riddable saddle} \cite{sun2015nonconvex} if the Hessian has a strictly negative eigenvalue, {i.e.,} $\lambda_{\min}(\nabla^2 \psi(\x))<0$. A twice continuously differentiable function satisfies the \emph{strict saddle property} if every critical point is either a local minimum or is a strict saddle \cite{ge2016matrix}.

Heuristically, the strict saddle property describes a geometric structure of the landscape: every non-local-minimum critical point is a strict saddle, where the Hessian has a strictly negative eigenvalue.
This property ensures that many local search algorithms, such as noisy gradient descent \cite{ge2015escaping} and the trust region method \cite{sun2016geometric},
can escape from all the saddles along the directions associated with the Hessian's negative eigenvalues, and hence converge to a local minimum:
\begin{theorem}[Escaping from saddles \cite{ge2015escaping,sun2015nonconvex,lee2016gradient}]\label{thm:informal}
(informal)
The strict saddle property allows many local search algorithms, such as noisy gradient descent and the trust region method, to converge to a local minimum.
\end{theorem}

Our main result builds on the assumption that the convex loss function $f(X)$ is $(2r,4r)$-restricted well-conditioned:
\begin{align}
&m\|D\|_F^2\leq [\nabla^2 f(X)](D,D)\leq M \|D\|_F^2,\ M/m\leq1.5 \text{ if } \tmop{rank}({X})\leq 2r\text{ and }\rank(D)\leq 4r.  \label{eqn:assumption}
\end{align}
We note that the assumption \eqref{eqn:assumption} is standard in matrix inverse problems \cite{bhojanapalli2015dropping}.
The main contribution of this work is establishing that under the restricted well-conditionedness of the convex loss function, the factored optimization \eqref{eqn:factored} has no spurious local minima and satisfies the strict saddle property.
\begin{theorem}[Strict saddle property]\label{thm:main}
Suppose the function $f(X)$ is  twice continuously differentiable and satisfies the $(2r,4r)$-restricted well-conditioned property \eqref{eqn:assumption}.
Assume $X^\star$ is an optimal solution of the optimization \eqref{eqn:origin} with $\rank(X^\star)= r^\star$. Set $r\geq r^\star$ in \eqref{eqn:factored}.
Let ${(U,V)}$ be any critical point of $g(U,V)$ satisfying $\nabla g(U,V)=\zero$. Then $(U,V)$  either corresponds to a  factorization  of  $X^\star$, {i.e.},
\begin{align}\label{eqn:global0}
X^\star=UV^T;
\end{align}
or is a strict saddle of the factored problem  \eqref{eqn:factored}. More precisely, denote $W=[U^T ~V^T]^T$. Then
\begin{align*}
&\lambda_{\min}(\nabla^2g(W))\leq
 \begin{cases}
-0.12m\min\{0.5\rho(W)^2,\rho(X^\star)\}, & \text{ if $r\geq r^\star$};\\
-0.099m\rho(X^\star), & \text{ if $r= r^\star$};\\
-0.12m\rho(X^\star), &  \text{ if $W= \zero$}.
\end{cases}
\end{align*}
Here $\rho(W)$ denotes $W$'s smallest nonzero singular value.
\end{theorem}
\begin{remark}
In addition to the strict saddle property, Theorem \ref{thm:main} also shows that there is no spurious local minimum. This allows  a number of iterative optimization methods \cite{ge2015escaping,sun2016geometric,lee2016gradient} to find $X^\star$ with random initialization.
\end{remark}
\begin{remark}
Theorem \ref{thm:main} establishes the strict saddle property for both over-parameterization ($r > r^\star$) and exact parameterization ($r = r^\star$). Thus,
as long as we know an upper bound on $r^\star$, many simple iterative algorithms can help to find the global optimizer $X^\star$.
\end{remark}

\begin{remark}
The main result only requires $f(X)$ to be restricted well-conditioned. Hence, in addition to those with quadratic objective functions \cite{bhojanapalli2015dropping,li2016,bhojanapalli2016lowrankrecoveryl},  a range of other low-rank matrix recovery problems are covered by our main theorem, including  $1$-bit matrix completion \cite{davenport20141}, robust principal component analysis (PCA) \cite{li2016robust},  Poisson PCA \cite{salmon2014poisson}, and other more general low-rank matrix problems \cite{udell2016generalized}.
\end{remark}

\subsection{Related Work}

This research is inspired by several previous works where nonconvex reformulations of various convex optimizations are proposed and analyzed \cite{bhojanapalli2016lowrankrecoveryl, bhojanapalli2015dropping, sun2016geometric,sun2015complete,li2016}. Some of the proposed algorithms require initializing the first iterate into the attraction basin of the global optima \cite{candes2015Wirtinger, tu2015low, bhojanapalli2015dropping}, while others have guaranteed convergence with random initializations \cite{ge2015escaping,sun2016geometric, sun2015complete}.  The latter is achieved by studying the (nonconvex) landscape of the optimizations' objective function. Our work falls into the second category.

The most related work is non-square matrix sensing from linear observations, which minimizes the factored quadratic objective function \cite{park2016non}.
The ambiguity in the factored parameterization
$$\phi(U,V)=\phi(UR,VR^{-1}) \text{ for all nonsingular }R$$
tends to make the factored quadratic objective function badly-conditioned, especially when the matrix $R$ or its inverse is close to being singular \cite{park2016non,li2016symmetry}. To overcome this problem, the regularizer
\begin{align}\label{eqn:equal_energy}
\Theta_E(U,V)=\|U^TU-V^TV\|_F^2
\end{align}
is proposed to ensure that $U$ and $V$ have almost equal energy \cite{tu2015low, park2016non,li2016symmetry}.
Our result shows that it is not necessary to introduce the extra regularization \eqref{eqn:equal_energy}. Indeed, the representation \eqref{eqn:nuclear:nvx} of the nuclear norm implicitly requires $U$ and $V$ to have equal energy. As a reformulation of the convex program \eqref{eqn:origin}, the nonconvex optimization \eqref{eqn:factored} inherits all its statistical performance.
Furthermore, by relating the first order optimality condition of the factored problem with the global optimality of the original nuclear-norm regularized convex program, our work provides a more transparent theoretical analysis that shows how the convex geometry is transformed into a nonconvex one.

In \cite{cabral2013unifying}, Cabral et al. worked on a similar problem and showed all global optima of \eqref{eqn:factored} corresponds to the solution of the convex program \eqref{eqn:origin}. The work \cite{haeffele2015global} applied the factorization approach to a more broad class of problems. When specialized to matrix inverse problems, their results show that any local minimizer $U$ and $V$ with zero columns is a global minimum for over-parameterization case, {i.e.}, $r>\rank(X^\star)$. However, there are no results discussing the existence of spurious local minima or the degenerate saddles in these previous work.  We extend their work and further prove that as long as the loss function $f(X)$ is restricted well-conditioned, all local minima are global minima and there are no degenerate saddles with no requirement on the size of the variables.

\subsection{Notations}
In this section, we collect notations used throughout the paper.
Denote $[n]=\{1,2,\ldots,n-1,n\}$. We reserve the symbols $\eye$ and  $\zero$  for the identity matrix and zero matrix/vector, respectively.
$\O_r = \{R \in \R^{r\times r}: RR^T = \eye_r\}$ represents the set of $r\times r$ real matrices.
Matrix norms, such as the spectral, nuclear, and Frobenius norms, are denoted by $\|\cdot\|$, $\|\cdot\|_*$ and $\|\cdot\|_F$, respectively.
The smallest nonzero singular value for any $X$ is denoted by
$\rho(X)$.

For any row-block matrix $W=[U^T~V^T]^T$, we denote $\wh{W}=[U^T~-V^T]^T$ by changing the sign of the second block of $W$.
The gradient of a scalar function $f(Z):\R^{m\times n}\to \R$  is an $m\times n$ matrix, whose $(i,j)$th element is $[\nabla f(Z)]_{i,j}= \frac{\partial f(Z)}{\partial Z_{ij}}$ for $i\in [m]$, $j\in[n]$. Meanwhile,  the gradient can be also viewed as a linear form $[\nabla f(Z)](G) = \lg \nabla f(Z), G\rg = \sum_{i,j}\frac{\partial f(Z)}{\partial Z_{ij}} G_{ij}$ for any $G \in \R^{m\times n}$. We can view the Hessian of $f(Z)$ as a $4$th order tensor of size $m\times n\times m\times n$, whose $(i,j,k,l)$th entry is $[\nabla^2 f(Z)]_{i,j,k,l}$ $=\frac{\partial^2 f(Z)}{\partial Z_{ij}\partial Z_{k,l} }$ for $i, k \in [m]$, $j,l\in[n]$. Similarly, we can also view the Hessian as a bilinear form defined via $[\nabla^2 f(Z)](G,H)=\sum_{i,j,k,l}\frac{\partial^2 f(Z)}{\partial Z_{ij}\partial Z_{kl} } G_{ij}H_{kl}$ for any $G,H\in\R^{m\times n}$. Yet another way to represent the Hessian is as an $mn\times mn$ matrix $[\nabla^2 f(Z)]_{i,j}=\frac{\partial^2 f(Z)}{\partial x_i\partial x_j}$ for $i,j\in[mn]$, where $x_i$ is the $i$th element of the vectorization of $Z$. We will use these representations interchangeably whenever the specific form can be inferred from context.

\section{Problem Formulation}
In this work, we consider the nuclear norm regularization \eqref{eqn:origin}:
\[
\minimize_{X\in\R^{p\times q}} f(X)+\lambda\|X\|_*,
\]
which is equivalent to \cite[page 8]{recht2010guaranteed}:
\begin{equation}\label{eqn:origin:2}
\begin{aligned}
&\minimize_{X\in\R^{p\times q}} f(X)+\frac{\lambda}{2}(\tr(\Phi)+\tr(\Psi))\\
&\st \begin{bmatrix} \Phi&X\\X^T&\Psi \end{bmatrix} \succeq 0
\end{aligned}
\end{equation}
We can enforce the PSD constraint implicitly using the fact that any PSD variable $Q$ can be reparameterized as
$Q=WW^T$. More precisely, let
\begin{align}\label{eqn:replace:Q}
\begin{bmatrix} \Phi&X\\X^T&\Psi \end{bmatrix}=\begin{bmatrix} U\\V\end{bmatrix}\begin{bmatrix} U^T &V^T\end{bmatrix},
\end{align}
implying
\begin{equation}
\begin{aligned}\label{eqn:observation1}
X=UV^T,\ \
\Phi=UU^T,\ \
\Psi=VV^T.
\end{aligned}
\end{equation}
Plugging \eqref{eqn:observation1} into \eqref{eqn:origin:2} gives the Burer-Monteiro factored reformulation \eqref{eqn:factored}:
\begin{align*}
\minimize_{U\in\R^{p\times r},V\in\R^{q\times r}} f(UV^T)+  \frac{\lambda}{2}(\|U\|_F^2+\|V\|_F^2),
\end{align*}
where the PSD constraint is dropped by construction.
As discussed in Section \ref{sec:introduction}, this new factored formulation \eqref{eqn:factored} can potentially increase computational efficiency
in two ways: {\em (i)} avoiding expensive SVDs because of replacing the nuclear norm $\|X\|_*$ with the squared term $(\|U\|_F^2+\|V\|_F^2)/2$; {\em (ii)}  a substantial reduction in the number of the optimization variables from $pq$ to $(p+q)r$.

\subsection{The Necessity of Restricted Well-Conditionedness }
The factored parameterization
\[\phi(U,V)=UV^T\]
transforms the original convex optimization into a nonconvex one and introduces additional critical points ({i.e.}, those $(U,V)$ with $\nabla g(U,V) = \zero$ that are not global optima of the factored optimization \eqref{eqn:factored}). This causes the convex program \eqref{eqn:origin} and its low-rank reformulation \eqref{eqn:factored} not equivalent.
In particular, when the loss function $f(X)$ is not well-conditioned, spurious local minima might emerge from these introduced critical points. For example, Srebro et al. \cite{srebro2003weighted} showed for weighted low-rank approximation, if the objective function is not well-conditioned ({e.g.}, the weight matrix has a few dominant entries), a non-global local minimum emerges for the factored problem. Similarly, as discussed in \cite[Examples 1,2]{bhojanapalli2016lowrankrecoveryl}, when the objective function of a matrix sensing problem does not satisfy the Restricted Isometry Property (RIP), there would be spurious local minima in the factored problem.
Therefore, to enable the two problems \eqref{eqn:origin} and \eqref{eqn:factored} equivalent, it is reasonable to introduce the restricted well-conditionedness assumption \eqref{eqn:assumption} for the general loss function $f(X)$ in \eqref{eqn:origin}.

\subsection{Consequences of  Restricted Well-Conditionedness }
We observe that the $(2r,4r)$-restricted well-conditionedness assumption \eqref{eqn:assumption}  reduces to the RIP when the objective function is quadratic \cite{bhojanapalli2016lowrankrecoveryl}. To see this,
we note the $(2r,4r)$-restricted well-conditionedness assumption \eqref{eqn:assumption}
indicates a restricted orthogonality property:
\begin{proposition}\label{pro:RIP}
Let $f(X)$ satisfies the $(2r,4r)$-restricted well-conditionedness assumption \eqref{eqn:assumption} with positive $m$ and $M$. Then
\begin{align}
&\left|\frac{2}{M+m}[\nabla^2f(X)](G,H) - \langle G,H \rangle\right| \leq \frac{M-m}{M+m}\|G\|_F \|H\|_F\nonumber
\end{align}
for any $p\times q$ matrices $X,G,H$ of rank at most $2r$.
\end{proposition}
\begin{proof}[Proof of Proposition~\ref{pro:RIP}]
The proof follows similar to~\cite{candes2008restricted}.
If either $G$ or $H$ is zero, it holds since both sides are $0$.
For nonzero $G$ and $H$, we can assume $\|G\|_F = \|H\|_F = 1$ without loss of generality. Then the assumption~\eqref{eqn:assumption} implies
\begin{align*}
&m \left\|G-H\right\|_F^2 \leq [\nabla^2 f(X)](G-H,G-H) \leq M \left\|G-H\right\|_F^2, \\
&m \left\|G+H\right\|_F^2 \leq [\nabla^2 f(X)](G+H,G+H) \leq M \left\|G+H\right\|_F^2.
\end{align*}
Thus we have
\begin{align*}
\left|2\left[\nabla^2f(X)\right](G,H) - (M+m)\left\langle G,H \right\rangle
\right|
&\leq \frac{M-m}{2} \underbrace{\left(\left\|G\right\|_F^2 +\left\|H\right\|_F^2\right)}_{=2}
= M-m=(M-m)\underbrace{\|G\|_F\|H\|_F}_{=1}.
\end{align*}
We complete the proof by dividing both sides by $M+m$:
\[\left|\frac{2}{M+m}[\nabla^2f(X)](G,H) - \langle G,H \rangle\right| \leq \frac{M-m}{M+m}\|G\|_F \|H\|_F.\]
\end{proof}

Sharing a similar spirit with the standard RIP, Proposition \ref{pro:RIP} also implies that the operator $\frac{2}{M+m}\nabla^2 f(X)$ preserves geometric structures, for low-rank matrices.
We intend to show that, when the function $f(X)$ in \eqref{eqn:origin} satisfy the restricted well-conditioned assumption \eqref{eqn:assumption}, the two programs \eqref{eqn:origin} and \eqref{eqn:factored} are equivalent: we can always find the global optimizer $X^\star$ by applying the simple iterative optimization methods to the factored problem.

\section{Understand the Transformed Landscape}
It is interesting to understand how the parameteriztion $\phi(U,V)$ transforms the geometric structures of the convex objective function $f(X)$  by categorizing  the critical points of the nonconvex factored function $g(U,V)$. In particular, we will illustrate how the globally optimal solution of the convex program is transformed in the domain of $g(U,V)$.
Furthermore, we will explore the properties of the additional critical points introduced by the parameterization and find a way of utilizing these properties to prove the strict saddle property.  For those purposes, the optimality conditions for the two programs \eqref{eqn:origin} and \eqref{eqn:factored} will be compared.

Before continuing this geometry-based argument, it is important to have a good understanding of the domain of the factored problem and set up a metric for this domain.
\subsection{Metric in the Domain of the Factored Problem}
Since the parameterization $\phi(U,V)$ and the factored regularization $\Theta(U,V)$ in the factored objective function $g(U,V)$ are both rotational invariant:
\begin{align*}
\phi(U,V)&=\phi(UR,VR) \text{ for }R\in\O_r \ \text{ and }\
\Theta(U,V)=\Theta(UR,VR) \text{ for }R\in\O_r,
\end{align*}
we obtain that $g(U,V)$ is a rotational-invariant function and the domain of  $g(U,V)$ is stratified into equivalent classes and can be treated as a quotient manifold \cite{absil2009optimization}.

For matrices lying in the same equivalent classes, they differ each other by only an orthogonal rotation.
Hence, to measure the distance of $W_1$ and $W_2$ lying in the quotient manifold, we can define the distance on their corresponding equivalent classes:
\begin{align*}
\dist(W_1,W_2)
=&\min_{R_1\in\O_r,R_2\in\O_r}\|W_1R_1-W_2 R_2\|_F
=\min_{R_1\in\O_r,R_2\in\O_r}\|W_1-W_2 R_2R_1^T\|_F
=\min_{R\in\O_r}\|W_1-W_2 R\|_F,
\end{align*}
where the second line follows from the rotation invariance of $\|\cdot\|_F$ and the third line follows from the property of the closed multiplicative operation for the orthogonal group $\mathcal{G}=\O_r$  \cite[Definition 7.2]{chirikjian2016harmonic}.

\subsection{Optimality Condition for the Convex Program}

As an unconstrained convex optimization, all critical points of \eqref{eqn:origin} are global optima and are characterized by the necessary and sufficient KKT condition \cite{boyd2004convex}:
\begin{align}\label{eqn:optimality:origin}
\nabla f(X^\star)\in-\lambda\partial\|X^\star\|_*,
\end{align}
where $\partial\|X^\star\|_*$ denotes the subdifferential (the set of subgradient) of the nuclear norm $\|X\|_*$ evaluated at $X^\star$.
The subdifferential  of the matrix nuclear norm is defined by
\begin{align*}
\partial \|X\|_* =\{&D \in \R^{p\times q}:\|R\|_* \geq \|X\|_*+\langle R - X, D\rangle,
 \text{all}\ R \in \R^{p\times q}\}.
\end{align*}
We have a more explicit characterization of the subdifferential of the nuclear norm using the singular value decomposition. More specifically, suppose $X = P\Sigma Q^T$ is the (compact) singular value decomposition of $X\in \R^{p \times q}$ with $P \in \R^{q\times r}, Q \in \R^{q\times r}$ and $\Sigma$ being an $r\times r$ diagonal matrix. Then the subdifferential of the matrix nuclear norm at $X$ is given by \cite[Equation (2.9)]{recht2010guaranteed}
\begin{align*}
\partial \|X\|_* = \{& PQ^T + W: P^T W=0, WQ=0, \|W\| \leq 1\}.
\end{align*}
By combining this representation of the subdifferential and the KKT condition \eqref{eqn:optimality:origin}, we present an equivalent expression for the optimality condition:
\begin{equation}\label{eqn:optimality:origin:1}
\begin{aligned}
\begin{cases}
\nabla f(X^\star) Q^\star =-\lambda P^\star\\
\nabla f(X^\star)^T P^\star =-\lambda Q^\star
\end{cases}
\text{and }\ \ \
\|\nabla f(X^\star)\| &\leq \lambda,
\end{aligned}
\end{equation}
where $Q^\star, P^\star$  denote respectively the right- and left- singular matrices in the compact SVD of $X^\star=P^\star\Sigma^\star Q^{\star T}$ with $P^\star \in \R^{q\times r^\star}, Q^\star \in \R^{q\times r^\star}$ and $\Sigma^\star$ being an $r^\star\times r^\star$ diagonal matrix where we assume $\rank(X^\star)=r^\star.$
Since we set $r\geq r^\star$ in the factored problem \eqref{eqn:factored},   in order to agree with the dimensions  in \eqref{eqn:factored}, we define the optimal factors $U^\star\in\R^{p\times r}$, $V^\star\in\R^{q\times r}$ as
\begin{equation}
\begin{aligned}\label{eqn:global:factor}
U^\star&=P^\star[\sqrt{\Sigma^\star}~\zero_{r^\star\times(r-r^\star)}] R;\\
V^\star&=Q^\star[\sqrt{\Sigma^\star}~\zero_{r^\star\times(r-r^\star)}] R;
\end{aligned}
\end{equation}
where $R\in\O_r$. Consequently, with the optimal factors $U^\star,V^\star$ defined in \eqref{eqn:global:factor}, we can rewrite the optimal condition \eqref{eqn:optimality:origin:1} as
\begin{equation}\label{eqn:optimality:origin:2}
\begin{aligned}
\begin{cases}
\nabla f(X^\star) V^\star=-\lambda U^\star\\
\nabla f(X^\star)^T U^\star=-\lambda V^\star\\
\end{cases}
\text{and }\ \ \
\|\nabla f(X^\star)\| &\leq \lambda.
\end{aligned}
\end{equation}
Stacking the two  variables $U^\star,V^\star$ into $W^\star=[U^{\star T}~V^{\star T}]^T$, we obtain a more concise and equivalent form of \eqref{eqn:optimality:origin:2}:
\begin{equation}\label{eqn:optimality:origin:3}
\begin{aligned}
\Xi(X^\star)W^\star&=\zero\ \ \ \text{ and }\ \ \
\|\nabla f(X^\star)\| \leq \lambda
\end{aligned}
\end{equation}
with
\begin{align}\label{eqn:Xi}
\Xi(X):=\begin{bmatrix}
\lambda\eye&\nabla f(X)\\
\nabla f(X)^T&\lambda\eye
\end{bmatrix}
\text{ for all }X.
\end{align}
An immediate result of \eqref{eqn:optimality:origin:3} is
\begin{align}\label{eqn:PSD}
\Xi(X^\star)
=\begin{bmatrix}
\lambda\eye&\nabla f(X^\star)\\
\nabla f(X^\star)^T&\lambda\eye
\end{bmatrix}
\succeq 0
\end{align}
by Schur complement theorem \cite[A.5.5]{boyd2004convex} in view of $\|\nabla f(X^\star)\|\leq\lambda$ in \eqref{eqn:optimality:origin:3}.

\subsection{Properties of the Critical Points of the Factored Program}
First of all, the gradient of $g(U,V)$ is given by
\begin{equation*}
\begin{aligned}
\nabla g(U,V)
=
\begin{bmatrix}
\nabla_U g(U,V)\\
\nabla_V g(U,V)
\end{bmatrix}
=
\begin{bmatrix}
\nabla f(X)V+\lambda U\\
\nabla f(X)^T U+\lambda V
\end{bmatrix}.
\end{aligned}
\end{equation*}
By invoking the notation $\Xi(X)$ in \eqref{eqn:Xi}, we obtain a more concise expression
\begin{equation}\label{eqn:grad:1}
\begin{aligned}
\nabla g(W)&=\Xi(X)W.
\end{aligned}
\end{equation}
Let the set of the critical points of $g(W)$ be denoted as
\begin{align}\label{eqn:critical}
\X:=\{W\in\R^{(p+q)\times r}: \nabla g(W)=\zero\}.
\end{align}
It is easy to see that any critical point $W$ of the factored problem \eqref{eqn:factored}, {i.e.}, $W\in\X$, also satisfies the left part of the optimality condition \eqref{eqn:optimality:origin:3} of the convex program. If the critical point $W=(U,V)$ additionally satisfies $\|\nabla f(UV^T)\|\leq\lambda$,  then the pair $(U,V)$ corresponds to the global optimizer $X^\star$. It remains to study the additional critical points that violates $\|\nabla f(UV^T)\|\leq\lambda$, which are introduced by the parameterization $\phi(U, V)$.

To show that the nuclear norm reformulation $(\|U\|_F^2 + \|V\|_F^2)/2$ guarantees $U$ and $V$ have equal energy at every critical point, we define the notation of balanced pairs:
\begin{definition}[Balanced pairs] \label{def:balanced:pair}
Let $U\in\R^{p\times r}$ and $V\in\R^{q\times r}$.
Then $(U,V)$ is a balanced pair if the Gram matrices of $U$ and $V$ are the same:
\[U^TU-V^TV=\zero.\]
All the balanced pairs form the balanced set, denoted by $\E(p,q,r)$ where $(p,q,r)$ indicates the dimensions of $(U,V)$.
\end{definition}
By stacking the variables $U,V$ into $W=[U^T~V^T]$ and invoking $\wh{W}=[U^T~-V^T]^T$, we get
\begin{align}\label{eqn:balanced:equation}
\wh{W}^T{W}=W^T\wh{W}=U^TU-V^TV.
\end{align}
Then the definition of the set $\E(p,q,r)$ simplifies as
\begin{align}\label{eqn:balanced:set}
\E(p,q,r)=\left\{W\in\R^{(p+q)\times r}: W^T\wh{W}=\zero\right\}.
\end{align}
Proposition \ref{pro:ef} claims that the critical points of $g(U,V)$ are balanced pairs,  whose proof is given in Appendix \ref{sec:proof:pro:ef}.
\begin{proposition}\label{pro:ef}
Let $\X$ be the set of critical points of $g(U,V)$ in \eqref{eqn:factored} and $\E$ be the balanced set \eqref{eqn:balanced:set}. Then we have
$\X\subset\E(p,q,r).$
\end{proposition}
Next, we derive two properties of those points lying in $\E(p,q,r)$ that involve the relationship of the energy of the on-diagonal blocks and the off-diagonal blocks of certain block matrix, say $Q=\begin{bmatrix} A &B\\C&D  \end{bmatrix}$ with
$A\in\R^{p\times p}, B\in\R^{p\times q}, C\in\R^{q\times p}, D\in\R^{q\times q}$. More precisely,  we  define
\begin{equation}
\begin{aligned}
\Pon(Q)=\begin{bmatrix} A&\zero  \\ \zero&D  \end{bmatrix} \text{ and }
\Poff(Q)=\begin{bmatrix} \zero&B \\ C&\zero  \end{bmatrix}.
\end{aligned}
\end{equation}
When $\Pon(\cdot)$ and $\Poff(\cdot)$ are acting on the product of two block matrices $W_1W_2^T$ for
$W_1=[U_1^T~V_1^T]^T$ and  $W_2=[U_2^T~V_2^T]^T$ with $U_1,U_2\in\R^{p\times r}$ and $V_1,V_2\in\R^{q\times r}$, we observe that
\begin{equation}\label{eqn:Pon:Poff}
\begin{aligned}
\Pon(W_1W_2^T)&=\begin{bmatrix} U_1U_2^T&\zero  \\ \zero& V_1V_2^T  \end{bmatrix}= \frac{W_1W_2^T+\wh{W}_1\wh{W}_2^T}{2};\\
\Poff(W_1W_2^T)&=\begin{bmatrix} \zero&U_1V_2^T  \\ V_1U_2^T& \zero \end{bmatrix}= \frac{W_1W_2^T-\wh{W}_1\wh{W}_2^T}{2}.
\end{aligned}
\end{equation}
Now we are ready to provide these two properties,  summarized in Lemma \ref{lem:ef:1}, \ref{lem:ef:2}, whose proofs are in Appendix \ref{sec:proof:lem:ef:1}, Appendix \ref{sec:proof:lem:ef:2}, respectively.
\begin{lem}\label{lem:ef:1} Let $W\in\E(p,q,r)$. Then for every $D=[ D_U^T~D_V^T]^T$ with consistent sizes,
we have
\[\|\Pon(DW^T)\|_F^2=\|\Poff(DW^T)\|_F^2.\]
\end{lem}

\begin{lem}\label{lem:ef:2}
Let $W_1,W_2\in\E(p,q,r)$ with $W_1=[U_1^T~V_1^T]^T$ and  $W_2=[U_2^T~V_2^T]^T$. Then
\[\|\Pon(W_1W_1^T-W_2 W_2^T)\|_F^2\leq\|\Poff(W_1W_1^T-W_2 W_2^T)\|_F^2. \]
\end{lem}

\subsection{Characterizing the Critical Points by the Hessian}

First of all, we observe that $W^\star=[U^{\star T}~V^{\star T}]^T$ with $(U^\star,V^\star)$  given in \eqref{eqn:global:factor}, is also the global optimum of the factored program \eqref{eqn:factored} given by Proposition \ref{pro:global} with its proof listed in Appendix \ref{sec:proof:pro:global}:
\begin{proposition}\label{pro:global}
For any  $(U^\star,V^\star)$  given in \eqref{eqn:global:factor}, we have
\[g(U^\star,V^\star)\leq g(U,V)\text{ for all }U\in\R^{p\times r}, V\in\R^{q\times r}\]
implying $(U^\star,V^\star)$ is also a global optimum of the factored program \eqref{eqn:factored}.
\end{proposition}
However, due to the nonconvexity of the factored problem, only characterizing the global optimizers is not sufficient. One should also eliminate possibility of the existence of spurious local minima or degenerate saddles.
For this purpose, we analyze the Hessian quadratic form of $g(W)$:
\begin{equation}\label{eqn:Hessian}
\begin{aligned}
&[\nabla^2 g(W)](D,D)=
\lg\Xi(X)
, DD^T\rg
+
[\nabla^2 f(X)](D_UV^T+UD_V^T,D_UV^T+UD_V^T).
\end{aligned}
\end{equation}
We intend to show the following: for any critical point $W\in\X$, if $W\neq W^\star$, we can find a direction $D=[D_U^T~D_V^T]^T$, along which the Hessian $\nabla^2 g(W)$ has a strictly negative curvature $[\nabla^2 g(W)](D,D)<-\tau \|D\|_F^2$ for some $\tau>0$.
We choose  $D$ as the direction from the $W$ to its closest globally optimal factor $W^\star R$ of the same size as $W$:
\[D=W-W^\star R,\]
with $R=\argmin_{\tilde{R}\in\O_r}\|W-W^\star \tilde{R}\|_F.$
Such a choice is inspired by
the previous work \cite{li2016}, particularly \cite[Example 1]{li2016}).

\section{Proof of Theorem \ref{thm:main}}
We choose $D$ as the direction from $W$ to its closest globally optimal factor: $D=W-W^\star R $ with $R=\argmin_{R}\|W-W^\star R\|_F$, and wish to show that, as long as $W\neq W^\star$,  the Hessian $\nabla^2 g(W)$ has a strictly negative curvature along $D$.
\subsection{Supporting Lemmas}
We start with several lemmas that will be used in the proof.
The first two lemmas bound the distance $\|W_1W_1^T-W_2W_2^T\|_F$ by  their square-root distance $\dist(W_1,W_2)$.
\begin{restatable}{lem}{lema}\cite[Lemma 3]{li2016}\label{lem:Gongguo}
Let $W_1, W_2$ be of the same size.  Then
\[\|W_1W_1^T-W_2W_2^T\|_F \geq \min\{\rho(W_1),\rho(W_2)\} \dist(W_1, W_2).\]
\end{restatable}

\begin{restatable}{lem}{lemben} \cite[Lemma 5.4]{tu2015low} \label{lem:ben}
Let $W_1, W_2$ be of the same size and $\rank(W_1)=r$.  Then
\[\|W_1W_1^T-W_2W_2^T\|_F \geq 2(\sqrt2-1)\rho(W_1) \dist(W_1, W_2). \]
\end{restatable}

Lemma \ref{lem:bound:WD} divides  $\|(W_1 - W_2)W_1^T \|_F^2 $ into two terms:   $\|W_1 W_1^T  - W_2W_2^T  \|_F^2  $ and $\|(W_1W_1^T  - W_2W_2^T ) {Q}{Q}^T \|_F^2$, where $QQ^T$ is the projector onto $\range(W)$. The key is letting  the first part $\|W_1 W_1^T  - W_2W_2^T  \|_F^2  $  have a small coefficient. Then Lemma \ref{lem:bound:QQ} further controls the second part, with the proof listed in Appendix \ref{app:lem:bound:QQ}.
\begin{lem}\cite[Lemma 4]{li2016}\label{lem:bound:WD}
 Let $W_1$ and $W_2$ be of the same size and
 $W_1^T  W_2 = W_2^T  W_1$ be PSD.
Assume ${Q}$ is an orthogonal matrix whose columns span $\range(W_1)$.
Then
\begin{align*}
&\|(W_1 - W_2)W_1^T \|_F^2 \leq (1/8)\|W_1 W_1^T  - W_2W_2^T  \|_F^2
+
(3 + 1/(2(\sqrt{2} -1)))\|(W_1W_1^T  - W_2W_2^T ) {Q}{Q}^T \|_F^2.
\end{align*}
\label{lem:2}
\end{lem}

\begin{lem}\label{lem:bound:QQ}
Suppose $f(X)$ satisfies \eqref{eqn:assumption}.
Let $W=[U^T~V^T]^T$ be any critical point of \eqref{eqn:factored},
$W^\star=[U^{\star T} ~V^{\star T}]^T$ correspond to the optimal solution of \eqref{eqn:origin}
and  $QQ^T$ be projection to $\range(W)$. Then
\begin{align*}
\|(WW^T-W^\star W^{\star T})QQ^T \|_F \leq  2\frac{M-m}{M+m}\|\phi(U,V)-X^\star\|_F.
\end{align*}
\label{lem:3}
\end{lem}

\subsection{A Formal Proof}
Let $D=W-W^\star R \text{ for }R=\argmin_{\tilde{R}}\|W-W^\star R\|_F.$
Denote   $\Gamma:=\int_{0}^1[\nabla^2 f(X^\star+t(X-X^\star))](X-X^\star)d t$ and invoke $\Xi(X):=\begin{bmatrix}
\lambda\eye&\nabla f(X)\\
\nabla f(X)^T&\lambda\eye
\end{bmatrix}
$ to simplify notations.
Then
\begin{align*}
&[\nabla^2 g(W)](D,D)=
\lg\Xi(X), DD^T\rg+[\nabla^2 f(X)](D_UV^T+UD_V^T,D_UV^T+UD_V^T)\\
&\stack{a}{=}
\lg\Xi(X)
, W^\star W^{\star T}-WW^T\rg+[\nabla^2 f(X)](D_UV^T+UD_V^T,D_UV^T+UD_V^T)\\
&\stack{b}{\leq}
\left\lg\Xi(X)-\Xi(X^\star)
, W^\star W^{\star T}-WW^T\right\rg+[\nabla^2 f(X)](D_UV^T+UD_V^T,D_UV^T+UD_V^T)\\
&\stack{c}{=}
\left\lg\begin{bmatrix}
\zero&\Gamma\\
\Gamma^T&\zero
\end{bmatrix}
, W^\star W^{\star T}-WW^T\right\rg+[\nabla^2 f(X)](D_UV^T+UD_V^T,D_UV^T+UD_V^T)\\
&\stack{d}{\leq} -2m\|X^\star-X\|_F^2+M\|D_UV^T+UD_V^T\|_F^2,
\end{align*}
where (a) follows from $\nabla g(W)=\Xi(X)W=\zero$ and \eqref{eqn:grad:1} and (b) holds since
$\lg \Xi(X^\star), W^\star W^{\star T}-WW^T\rg\leq0$
by $ \Xi(X^\star) W^\star=\zero$ in \eqref{eqn:optimality:origin:3} and $  \Xi(X^\star)\succeq 0$ in \eqref{eqn:PSD}.  (c)
follows from the integral form of the mean value theorem for vector-valued functions (see \cite[Eq. (A.57)]{nocedal2006numerical}). (d) follows from the restricted well-conditionedness assumption  \eqref{eqn:assumption} since $\rank(X^\star+t(X-X^\star))\leq 2r$, $\rank(X-X^\star)\leq 4r$ and $\rank(D_UV^T+UD_V^T)\leq 4r.$

Then
\begin{align*}
& -2m\|X^\star-X\|_F^2+M\|D_UV^T+UD_V^T\|_F^2\\
&\stack{a}{\leq} -0.5m \|WW^T-W^\star W^{\star T}\|_F^2+2M(\|D_UV^T\|_F^2+\|UD_V^T\|_F^2)\\
&\stack{b}{=} -0.5m \|WW^T-W^\star W^{\star T}\|_F^2+M\|DW^T\|_F^2\\
&\stack{c}{\leq}  \left(-0.5{m}+{M}/{8}+ 4.208 M \left(\frac{M-m}{M+m}\right)^2 \right)
\|WW^T-W^\star W^{\star T}\|_F^2
\\
&\stack{d}{\leq} -0.06m\|WW^T-W^\star W^{\star T}\|_F^2
\stackrel{~~~}{\leq}
\begin{cases}
-0.06m\min\{\rho(W)^2,\rho(W^\star)^2\}\|D\|_F^2  &\text{(Lemma \ref{lem:Gongguo})};
\\
-0.0495m\rho(W^\star)^2\|D\|_F^2  &\text{(Lemma \ref{lem:ben})};
\\
-0.06m\rho(W^\star)^2\|D\|_F^2  &\text{($W=\zero$)};
\end{cases}
\end{align*}
where (a) follows from Lemma \ref{lem:ef:2} and the fact $(x+y)^2\leq 2(x^2+y^2)$. (b) follows from Lemma \ref{lem:ef:1}.
For (c) to hold,  we first use  Lemma \ref{lem:bound:WD} to bound $\|DW^T\|_F^2=\|(W-W^\star R)W^T\|_F^2$ since $W^TW^\star\succeq0$ by \cite[Lemma 1]{li2016}. Then use Lemma \ref{lem:bound:QQ} to further bound $\|(W^\star-W)QQ^T\|_F^2$.
(d) holds when $M/n\leq1.5$.
For the last inequality, we apply Lemma \ref{lem:Gongguo}
when $r\geq\rank(W^\star)$ and  Lemma \ref{lem:ben}  when $r=\rank(W^\star)$ and observe $D=-W^\star R$ when $W=\zero$.
Finally, the proof follows from the fact that by \eqref{eqn:global:factor}, we have
$$
W^\star = \begin{bmatrix} P^\star\sqrt{\Sigma^\star}R \\ Q^\star\sqrt{\Sigma^\star}R \end{bmatrix} = \begin{bmatrix}P^\star/\sqrt{2} \\ Q^\star/\sqrt{2} \end{bmatrix} \left(\sqrt{2\Sigma^\star}\right)R\ \ \Longrightarrow\ \ \sigma_\ell(W^\star)=\sqrt{2\sigma_\ell(X^\star)}.
$$

\section{Conclusion}
In this work, we considered the minimization of a general convex loss function $f(X)$ regularized by the matrix nuclear norm $\|X\|_*$. To improve computational efficiency,  we applied the Burer-Monteiro factored formulation and showed that,
 as long as the convex function $f(X)$ is (restricted) well-conditioned, the factored problem has the following benign landscape: each critical point either produces a global optimum of the original convex program, or is a strict saddle where the Hessian matrix has a strictly negative eigenvalue. Such geometric structure then allows many iterative optimization methods to escape from the saddles and thus converge to a global minimizer with random initializations.

\appendix

\section{Proof of Lemma \ref{lem:ef:1}} \label{sec:proof:lem:ef:1}

First, by \eqref{eqn:Pon:Poff}, we can get
$$\|\Pon(DW^T)\|_F^2-\|\Poff(DW^T)\|_F^2
= \lg DW^T, \wh{D}\wh{W}^T \rg
= \lg \wh{D}^TD , \wh{W}^TW \rg
=0,$$
   since  $\wh{W}^TW =\zero$ by \eqref{eqn:balanced:set}.

\section{Proof of Lemma \ref{lem:ef:2}} \label{sec:proof:lem:ef:2}
Replacing $\Pon(\cdot)$ and $\Poff(\cdot)$ by \eqref{eqn:Pon:Poff} and expanding $\|\cdot\|_F^2$ by the innerproducts, we have
\begin{align*}
&\|\Pon(W_1W_1^T-W_2 W_2^T)\|_F^2-\|\Poff(W_1W_1^T-W_2 W_2^T)\|_F^2\\
&=\lg W_1W_1^T-W_2 W_2^T,\wh{W}_1\wh{W}_1^T-\wh{W}_2 \wh{W}_2^T\rg\\
&=\lg W_1W_1^T,\wh{W}_1\wh{W}_1^T\rg+\lg W_2  W_2^T,\wh{W}_2  \wh{W}_2^T\rg
-\lg W_1W_1^T,\wh{W}_2 \wh{W}_2^T\rg-\lg \wh{W}_1\wh{W}_1^T,W_2  W_2^T\rg\\
&=-\lg W_1W_1^T,\wh{W}_2 \wh{W}_2^T\rg-\lg \wh{W}_1\wh{W}_1^T,W_2  W_2^T\rg
\leq 0,
\end{align*}
where the third equality follows from \eqref{eqn:balanced:set} and the last inequality holds by recognizing these PSD matrices: $W_1W_1^T$, $\wh{W}_1\wh{W}_1^T$, $W_2W_2^T$ and $\wh{W}_2\wh{W}_2^T$.

\section{Proof of Lemma \ref{lem:bound:QQ}}\label{app:lem:bound:QQ}
To simplify notations, we denote $\PP_W=QQ^T$ and invoke $\Xi(X):=\begin{bmatrix}
\lambda\eye&\nabla f(X)\\
\nabla f(X)^T&\lambda\eye
\end{bmatrix}$. Then
\begin{align*}
&\Rightarrow
\Xi(X)
W=\zero
\Rightarrow
\left\lg
\Xi(X)
,
ZW^T
\right\rg
=0,\forall Z=(Z_U,Z_V)
\\
&\Rightarrow
\left\lg
\Xi(X)-\Xi(X^\star)
+
\Xi(X^\star)
,
ZW^T
\right\rg
=0,\forall Z
\\
&\stack{a}{\Rightarrow}
\left\lg
\begin{bmatrix}
\zero&\int_{0}^1[\nabla^2 f(X^\star+t(X-X^\star))](X-X^\star)d t\\
*&\zero
\end{bmatrix}
+
\Xi(X^\star)
,
ZW^T
\right\rg
=0,\forall Z
\\
&\Rightarrow
\int_{0}^1[\nabla^2 f(X(t))](X-X^\star,Z_UV^T+UZ_V^T)d t
+
\left\lg
\Xi(X^\star)
,
ZW^T
\right\rg
=0, \forall Z
\end{align*}
where  (a) follows from the integral form of the mean value theorem for vector-valued functions (see \cite[Eq. (A.57)]{nocedal2006numerical}).
Then, in view of Proposition \ref{pro:RIP} and \eqref{eqn:Pon:Poff}, we arrive at
\begin{equation}\label{eqn:A:1}
\begin{aligned}
&\left|\frac{2}{M+m}\underbrace{\left\lg
\Xi(X^\star)
,
ZW^T
\right\rg}_{\Pi_1(Z)}+ \underbrace{\lg\Poff(WW^T-W^\star W^{\star T}),ZW^T\rg}_{\Pi_2(Z)}\right|
\leq \frac{M-m}{M+m}
\|X-X^\star\|_F\underbrace{\|\Poff(ZW^T)\|_F}_{\Pi_3(Z)},\forall Z.
\end{aligned}
\end{equation}
When $Z=(WW^T - W^\star W^{\star T}){W^T}^{\dagger}$,  we further can show
\begin{align}
\Pi_1(Z)&\geq0 \label{eqn:Pi:1};\\
\Pi_2(Z)&\geq\frac{1}{2}\|(WW^T-W^\star W^{\star T})\PP_W\|_F^2\label{eqn:Pi:2};\\
\Pi_3(Z)& \leq \|(WW^T-W^\star W^{\star T})\PP_W)\|_F\label{eqn:Pi:3}.
\end{align}
Finally, we complete the proof by combining \eqref{eqn:A:1} with \eqref{eqn:Pi:1},\eqref{eqn:Pi:2},\eqref{eqn:Pi:3} to get
\begin{align*}
&\frac{1}{2}\|(WW^T-W^\star W^{\star T})\PP_W\|_F^2
\leq \frac{M-m}{M+m}
\|X-X^\star\|_F\|(WW^T-W^\star W^{\star T})\PP_W\|_F.
\end{align*}

\noindent{\bf Show \eqref{eqn:Pi:1}.} First note that $ZW^T=(WW^T-W^\star W^{\star T})\PP_W$ when $Z=(WW^T - W^\star W^{\star T}){W^T}^{\dagger}$. Then
\[\Pi_1(Z)=\lg\Xi^\star,(WW^T-W^\star W^{\star T})\PP_W\rg=\lg\Xi^\star,WW^T\rg\geq0,\]
where the second equality holds since $WW^T\PP_W=WW^T$ and $\Xi^\star W^\star=\zero$ by \eqref{eqn:optimality:origin:3}. The inequality is due to \eqref{eqn:PSD}.

\noindent{\bf Show \eqref{eqn:Pi:2}.}
First recognize that $\Poff(WW^T-W^\star W^{\star T})=\frac{1}{2}( WW^T-W^\star W^{\star T}-\wh{W}\wh{W}^T+\wh{W}^\star \wh{W}^{\star T}).$
\begin{align*}
\Pi_2(Z)&=\lg\Poff(WW^T-W^\star W^{\star T}),ZW^T\rg\\
&= \frac{1}{2}\lg WW^T-W^\star W^{\star T}, (WW^T-W^\star W^{\star T})\PP_W\rg
- \frac{1}{2}\lg \wh{W}\wh{W}^T-\wh{W}^\star \wh{W}^{\star T}, (WW^T-W^\star W^{\star T})\PP_W\rg.
\end{align*}
Then  \eqref{eqn:Pi:2} follows from
\begin{align*}
&\lg \wh{W}\wh{W}^T-\wh{W}^\star \wh{W}^{\star T}, (WW^T-W^\star W^{\star T})\PP_W\rg
=\lg \wh{W}\wh{W}^T,-W^\star W^{\star T}\rg+\lg -\wh{W}^\star \wh{W}^{\star T},WW^T\rg
\leq0,
\end{align*}
where the first equality follows from \eqref{eqn:balanced:set} and the inequality holds by recognizing those PSD matrices.

\noindent{\bf Show \eqref{eqn:Pi:3}.}
By plugging $Z=(WW^T - W^\star W^{\star T}){W^T}^{\dagger}$, we get
$\Pi_3(Z)=\|\Poff((WW^T-W^\star W^{\star T})\PP_W)\|$
which is obviously no larger than $\|(WW^T-W^\star W^{\star T})\PP_W\|_F.$

\section{Proof of Proposition \ref{pro:ef}} \label{sec:proof:pro:ef}
Note $g(W)=\zero$ for $W\in\X$. Hence
\[\wh{W}^T\nabla g(W)+\nabla g(W)^T\wh{W}=\zero\]
implying
$\wh{W}^T W+W^T\wh{W} =\zero$ by \eqref{eqn:grad:1}, which finishes the proof  since $\wh{W}^TW=W^T\wh{W}$ by
\eqref{eqn:balanced:equation}.

\section{Proof of Proposition \ref{pro:global}} \label{sec:proof:pro:global}
First of all, by \eqref{eqn:global:factor}, we have
\[
\Theta(U^\star,V^\star)
=\frac{1}{2}(\|U^\star\|_F^2+\|V^\star\|_F^2)
=\frac{1}{2}(\|\sqrt{\Sigma^\star}\|_F^2+\|\sqrt{\Sigma^\star}\|_F^2)
=\|\sqrt{\Sigma^\star}\|_F^2=\|X^\star\|_*.
\]
Thus,
\begin{align*}
f(\phi(U^\star,V^\star))+\lambda\Theta(U^\star,V^\star)
&=f(X^\star)+\lambda\|X^\star\|_*\\
&\leq f(X)+\lambda\|X\|_*\\
&= f(\phi(U,V))+\lambda\|\phi(U,V)\|_*\\
&\leq f(\phi(U,V))+\lambda\Theta(U,V),
\end{align*}
where the first inequality follows from the optimality of $X^\star$ for \eqref{eqn:origin}. The second equality holds by choosing $X=UV^T.$ Finally,  the last inequality holds since $\|\phi(U,V)\|_*\leq \Theta(U,V)$ by \eqref{eqn:nuclear:nvx}.

\bibliographystyle{plain}
\bibliography{nonconvex}

\end{document}